\documentclass[12pt, a4paper]{amsart}
\usepackage[T1]{fontenc}
\usepackage[utf8]{inputenc}
\usepackage[english]{babel}

\usepackage{amsfonts}
\usepackage{amsthm}
\usepackage{xcolor}
\usepackage{hyperref}

\numberwithin{equation}{section}
\theoremstyle{plain}
\newtheorem{thm}{Theorem}[section]

\newtheorem{coro}[thm]{Corollary}

\newtheorem{thmA}{Theorem}
\newtheorem{conjectureA}[thmA]{Conjecture}

\theoremstyle{definition}

\newcommand{\N}{\mathbb{N}}
\newcommand{\Z}{\mathbb{Z}}
\newcommand{\C}{\mathbb{C}}
\newcommand{\h}{\mathcal{H}}
\newcommand{\Hol}{\mathcal{H}(\D)}
\newcommand{\D}{\mathbb{D}}

\newcommand{\T}{\mathbb{T}}

\usepackage{cite}

\title[Contractive inequalities]{Contractive inequalities between Dirichlet and Hardy spaces}
\author{Adrián Llinares}

\address{Adrián Llinares, Department of Mathematical Sciences, Norwegian University of Science and Technology, NO-7491 Trondheim, Norway}
\email{adrian.llinares@ntnu.no}


\date{\today}

\begin{document}

    \begin{abstract}
        We prove a conjecture of Brevig, Ortega-Cerdà, Seip and Zhao about contractive inequalities between Dirichlet and Hardy spaces and discuss its consequent connection with the Riesz projection.
    \end{abstract}

    \maketitle
    
    \section{Introduction}
    
        As usual, let $\D$ and $\T$ be the open unit disk in the complex plane $\C$ and its boundary, respectively. Given an holomorphic function $f$ in $\D$ and an exponent $p > 0$, $M_p(r,f)$ will denote its integral $p$-mean on the circle centered at the origin with radius $r$, 
            \[
                M_p(r,f) := \left ( \dfrac{1}{2\pi} \int_0^{2\pi} |f(re^{it})|^p \: dt \right)^\frac{1}{p}, \quad 0 \leq r < 1.
            \]
            
        We will say that $f$ belongs to the \emph{Hardy space} $H^p$ if its $M_p$ means are uniformly bounded with respect to the radius. It is well-known that the radial limit $f(e^{it})~:=~\displaystyle{\lim_{r \rightarrow 1^-} f(re^{it})}$ exists for almost every $t \in [0, 2\pi]$ if $f \in H^p$. Moreover, these radial limits are $p$-integrable on $\T$, so we can define the $H^p$-norm as follows:
            \[
                \| f \|_{H^p} := \left ( \dfrac{1}{2\pi} \int_0^{2\pi} | f(e^{it})|^p \: dt \right)^\frac{1}{p}.
            \]
            
        Of course $\| \cdot \|_{H^p}$ is a proper norm if $p \geq 1$. Furthermore, in this case $H^p$ can be identified with the subspace of $L^p(\T)$ functions with vanishing negative Fourier coefficients. That is,
            \[
                H^p = \left \{ f \in L^p (\T) \: : \: \widehat{f} (k) = 0, \forall k < 0 \right \},
            \]
            where
            \[
                \widehat{f}(k) := \dfrac{1}{2\pi} \int_0^{2\pi} f(e^{it}) e^{-ikt} \: dt, \quad k \in \Z.
            \]
        
        For more information regarding the elementary properties of Hardy spaces, we refer the reader to \cite{MR0268655}.

        Let $dA(re^{it}) := \frac{r}{\pi} \: dr \, dt$ be the normalized area measure of $\D$. For $0~<~p$ and $\alpha > -1$, we define the \emph{standard weighted Bergman space} $A^p_\alpha$ as the set of all analytic functions whose moduli are $p$-integrable with respect to the measure $(1 - |z|^2)^\alpha \: dA(z)$. If $\alpha = 0$, we will simply write $A^p$ instead of $A^p_0$. It can be checked that, for a fixed $f$, the quantities
            \[
                \| f \|_{A^p_\alpha} := \left( (\alpha + 1) \int_\D (1 - |z|^2)^\alpha |f(z)|^p \: dA(z) \right)^\frac{1}{p}
            \]
            converge to $\| f \|_{H^p}$ when $\alpha$ approaches $-1$. Thus, we will understand $A^p_{-1}$ as $H^p$.
        
        For $p = 2$, $A^2_\alpha$ is a Hilbert space and $\| f \|_{A^2_\alpha}$ can be computed in terms of the Taylor coefficients of $f$. Specifically, we have that
            \[
                \| f \|_{A^2_\alpha} = \left ( \sum_{n = 0}^\infty \dfrac{|a_n|^2}{c_{\alpha + 2} (n)} \right)^\frac{1}{2},
            \]
            where $(1 - z)^{-\beta} = \displaystyle{\sum_{n \geq 0} c_\beta (n) z^n}$ and $\displaystyle{f(z) = \sum_{n \geq 0} a_n z^n}$.

        As a companion of $A^2_\alpha$, we introduce the \emph{weighted Dirichlet space} $D_\beta$ as the set of all holomorphic functions $f$ such that
            \[
                \| f \|_{D_\beta} := \left ( \sum_{n = 0}^\infty c_\beta (n) |a_n|^2 \right )^\frac{1}{2}
            \]
            is finite.

        Contractive inclusions between spaces of analytic functions have attracted the attention of the experts because of their multiple applications. For this work is especially relevant the following inequality, which was conjectured by Brevig, Ortega-Cerdà, Seip and Zhao \cite{MR3858278} and Lieb and Solovej \cite{MR4331821} and recently proved by Kulikov \cite{MR4472587}.
        
        \begin{thmA}\label{Kulikov}
            Let $0 < p < q$ and $-1 \leq \alpha < \beta$ such that $\frac{\alpha + 2}{p} = \frac{\beta + 2}{q}$. Then we have that
                \begin{equation} \label{KulikovDisp}
                    \| f \|_{A^q_\beta} \leq \| f \|_{A^p_\alpha}, \quad \forall f \in A^p_\alpha,
                \end{equation}
                and equality is possible if and only if there exists $\zeta \in \D$ and $C \in \C$ such that
                \begin{equation} \label{ReproducingKernel}
                    f(z) = C \dfrac{(1 - |\zeta|^2)^{\frac{\alpha + 2}{p}}}{(1 - \overline{\zeta}z)^{\frac{2(\alpha + 2)}{p}}}, \quad z \in \D.
                \end{equation}
        \end{thmA}

        It is worth mentioning that, for $\alpha = -1$ and $q = kp$ being $k$ a positive integer, the inequality \eqref{KulikovDisp} was already known because of the works of Carleman \cite{MR1544458} and Burbea \cite{MR882113}.

        Although it was not explicitly stated in \cite{MR4472587}, from Theorem \ref{Kulikov} we can derive a complete characterization of contractive inclusions between weighted Bergman spaces.
        
        \begin{coro}\label{CharacBergmanContractive}
            Assume that $A^p_\alpha \subset A^q_\beta$ for some $p, q \in (0,\infty)$ and $\alpha, \beta \geq -1$. Then the inclusion operator $\iota : A^p_\alpha \rightarrow A^q_\beta$ is contractive if and only if $p < q$ or $p \geq q$ and $\alpha \leq \beta$. 
        \end{coro}

        \begin{proof}
            The case $q \leq p$ was already known (see for example \cite{MR4355935} for a complete characterization for mixed norm spaces). If $p < q$, then the condition $A^p_\alpha~\subset~A^q_\beta$ yields that $\frac{\alpha + 2}{p} \leq \frac{\beta + 2}{q}$ \cite[Theorem 69]{MR2537698} and consequently
                \[
                    \| f \|_{A^q_\beta} \leq \| f \|_{A^q_{\frac{q}{p}(\alpha + 2) - 2}} \leq \| f \|_{A^p_\alpha}, \quad \forall f \in A^p_\alpha,
                \]
                proving that the inclusion has norm equal to 1 in this case too.
        \end{proof}

        For a suitable choice of $p$, $q$, $\alpha$ and $\beta$, from \eqref{KulikovDisp} we can see that
            \[
                \left ( \sum_{n = 0}^\infty \dfrac{|a_n|^2}{c_{\frac{2}{p}}(n)}\right)^\frac{1}{2} \leq \| f \|_{H^p},
            \]
            for all $f \in \Hol$ and $p \in (0, 2]$. For $2 < p$, Brevig, Ortega-Cerdà, Seip and Zhao \cite{MR3858278} conjectured the following contractive inequality:
        
        \begin{conjectureA} \label{BOSZ}
            If $p > 2$, then the inequality
                \begin{equation} \label{BOSZDisp}
                    \| f \|_{H^p} \leq \| f \|_{D_{\frac{p}{2}}}
                \end{equation}
                holds for all $f$ analytic function in $\D$.
        \end{conjectureA}

        The main goal of this paper is to prove Conjecture \ref{BOSZ} and to show its connection with the Riesz projection.

    \section{Proof of Conjecture \ref{BOSZ}}
    
        For the sake of clarity, $C_{p,n}$ will denote the restriction of the inclusion operator from $D_{\frac{p}{2}}$ into $H^p$ to the subspace of polynomials of degree $n$. That is,
            \[
                C_{p,n} := \sup \left \{ \left \| \sum_{k = 0}^n a_k z^k \right \|_{H^p} \: : \: \sum_{k = 0}^n c_{\frac{p}{2}} (k) | a_k |^2 = 1 \right \}.
            \]

        Clearly, \eqref{BOSZDisp} holds for any holomorphic function $f$ if and only if $C_{p,n} = 1$ for all $n \geq 1$. For a fixed $n$, it is trivial that there exists a polynomial $q_n$ of degree at most $n$ such that $\| q_n \|_{D_{\frac{p}{2}}} = 1$ and $\| q_n \|_{H^p} = C_{p,n}$. We are going to prove that such polynomials must be unimodular constants.

        \begin{thm} \label{MainTheo}
            Let $p > 2$. Then, for every $n \geq 1$ we have that $C_{p,n} = 1$. Moreover, if $q$ is a polynomial then the identity $\| q \|_{H^p} = \| q \|_{D_{\frac{p}{2}}}$ is possible if and only if $q$ is constant. In particular, Conjecture \ref{BOSZ} is true.
        \end{thm}

        \begin{proof}
            Let $n \geq 1$ and take $q_n$ a polynomial of degree at most $n$ satisfying that $\| q_n \|_{D_{\frac{p}{2}}} = 1$ and $\| q_n \|_{H^p} = C_{p,n}$. Consider $r \in (0, 1)$ and let $q_r(z)~:=~q_n(rz)$. It is clear that $q_r$ is also a polynomial of degree at most $n$, so $\| q_r \|_{H^p} \leq C_{p,n} \| q_r \|_{D_{\frac{p}{2}}}$.

            On the one hand,
                \[
                    \| q_r \|_{D_{\frac{p}{2}}}^2 = \sum_{k = 0}^n c_{\frac{p}{2}} (k) r^{2k} |a_k|^2 = 1 - \sum_{k = 1}^n c_{\frac{p}{2}} (k) (1 - r^{2k}) |a_k|^2
                \]
                and therefore
                \[
                    \| q_r \|_{D_{\frac{p}{2}}}^p = 1 - \dfrac{p}{2} \sum_{k = 1}^n c_{\frac{p}{2}} (k) (1 -r^{2k}) |a_k|^2 + o (1 - r), \quad r \rightarrow 1^-.
                \]

            On the other hand, it is obvious that $\| q_r \|_{H^p}^p = M_p^p(r,q_n)$. Thus, we have that
                \[
                    M_p^p(r,q_n) - C_{p,n}^p \leq - C_{p,n}^p \dfrac{p}{2} \sum_{k = 1}^n c_{\frac{p}{2}} (k) (1 - r^{2k}) | a_k |^2 + o(1 - r), \quad r \rightarrow 1^-,
                \]
                and then, we see that
                \[
                    \dfrac{M_p^p (1, q_n) - M_p^p (r, q_n)}{1 - r} \geq C_{p,n}^p \dfrac{p}{2} \sum_{k = 1}^n c_{\frac{p}{2}} (k) \dfrac{1 - r^{2k}}{1 - r} |a_k|^2 + \dfrac{o(1 - r)}{1 - r}, \quad r \rightarrow 1^-.
                \]
                
             Taking $r \rightarrow 1^-$, the \emph{Hardy--Stein identity} \cite[Theorem~2.18]{MR4321142}
                \[
                   \dfrac{d}{dr} M_p^p (r, f) = \dfrac{p^2}{2r} \int_{r\D} |f'(z)|^2 |f(z)|^{p-2} \: dA(z), \quad f \in \Hol,
                \]
                yields that
                \begin{equation} \label{KeyProperty}
                    \dfrac{p}{2} \int_\D |q_n'(z)|^2 |q_n(z)|^{p-2} \: dA(z) \geq C_{p,n}^p \sum_{k = 1}^n  k c_{\frac{p}{2}} (k) | a_k |^2 .
                \end{equation}

            We are going to show that \eqref{KeyProperty} is possible if and only if $q_n$ is constant. To this end, assume that $a_{k_0} \neq 0$ for some $k_0 \in \{1, \ldots, n\}$. Since $1 = \frac{p + 2}{2p} + \frac{p-2}{2p}$, Hölder's inequality implies that
                \[
                    \int_\D |q_n'(z)|^2 |q_n (z)|^{p-2} \: dA(z) \leq \| q_n' \|_{A^\frac{4p}{p+2}}^2 \| q_n \|_{A^{2p}}^{p-2},
                \]
                and because of Theorem \ref{Kulikov}, we deduce that
                \[
                    \int_\D |q_n'(z)|^2 |q_n (z)|^{p-2} \: dA(z) < \| q_n' \|^2_{A^2_{\frac{2}{p}-1}} \| q_n \|_{H^{p}}^{p-2} = C_{p,n}^{p-2} \sum_{k = 1}^n \dfrac{k^2}{c_{\frac{2}{p} + 1} (k-1)} |a_k|^2,
                \]
                since $q_n$ is not one of the reproducing kernels \eqref{ReproducingKernel}. 

                Summing up, we have that
                \begin{equation} \label{Contradiction}
                    C_{p,n}^2 \sum_{k = 1}^n k c_{\frac{p}{2}} (k) | a_k |^2 < \dfrac{p}{2} \sum_{k = 1}^n \dfrac{k^2}{c_{\frac{2}{p} + 1} (k-1)} |a_k|^2.
                \end{equation}
                However, it is easy to check that
                \[
                    \dfrac{p}{2} \dfrac{k}{c_{\frac{2}{p} + 1} (k-1)} \leq c_{\frac{p}{2}} (k), \quad \forall k \geq 1.
                \]
                Indeed, this inequality is a consequence of the monotonicity of the sequence 
                \[
                    A_k := \dfrac{c_{\frac{p}{2}} (k) c_{\frac{2}{p} + 1} (k - 1)}{k} = \dfrac{\Gamma \left( k + \frac{p}{2} \right) \Gamma \left( k + \frac{2}{p} \right)}{\Gamma \left( \frac{p}{2} \right) \Gamma \left( \frac{2}{p} + 1 \right) k!^2},
                \]
                which is increasing since
                \[
                    \dfrac{A_{k + 1}}{A_k} = \dfrac{\left( k + \frac{p}{2} \right) \left( k + \frac{2}{p}\right)}{(k + 1)^2} = \dfrac{k^2 + \left( \frac{p}{2} + \frac{2}{p} \right)k + 1}{k^2 + 2k + 1} > 1.
                \]
                
                Then, we see that \eqref{Contradiction} is not possible because $a_{k_0}$ is not zero. Thus, $q_n$ is constant and $C_{p,n} = 1$.
        \end{proof}

        \begin{coro}
            Let $p > 2$ and $\beta > 0$. Then $\iota : D_\beta \rightarrow H^p$ is contractive if and only if $\beta \geq \frac{p}{2}$.
        \end{coro}

        \begin{proof}
            The necessity of the condition $\beta \geq \frac{p}{2}$ was already noticed in \cite[Section~4]{MR3858278}. Theorem \ref{MainTheo} and the monotonicity of the weight $\{ c_\beta(n) \}_{n \geq 0}$ with respect to $\beta$ yield its sufficiency.
        \end{proof}

    \section{Extension to several variables and Dirichlet series} \label{SectionDirichlet}

        In this section we are going to show that the inequality
            \[
                \| f \|_{H^p} \leq \| f \|_{D_\frac{p}{2}}
            \]
            naturally induces an analogous inequality in the setting of spaces of Dirichlet series, in the same way that Theorem \ref{Kulikov} extended Helson's inequality \cite{MR2263964},
                \[
                    \left ( \sum_{k = 1}^n \dfrac{|a_k|^2}{d_2(k)} \right)^\frac{1}{2} \leq \lim_{T \rightarrow \infty} \dfrac{1}{2T} \int_{-T}^T \left| \sum_{k = 1}^n \dfrac{a_k}{k^{it}}\right| \: dt,
                \]
                to further Hardy spaces of Dirichlet series.

        For this purpose, we will need the following notation. For $d \geq 1$ and $p > 2$, we consider the spaces 
            \[
                H^p (\T^d) := \left \{ f \in L^p (\T^d) \: : \: \tilde{f} (k_1, \ldots, k_d) = 0 \mbox{ if } k_j < 0 \mbox{ for some } j \right \}
            \]
            and
            \[
                D_\frac{p}{2} (\D^d) := \left \{ f \in \mathcal{H} (\D^d) \: : \: \sum_{k_1, \ldots, k_d \geq 0} |a_{k_1, \ldots, k_d}|^2 c_\frac{p}{2}(k_1) \ldots c_\frac{p}{2} (k_d) < \infty \right\}.
            \]
            
        Let $\| \cdot \|_{H^p(\T^d)}$ and $\| \cdot \|_{D_\frac{p}{2}(\D^d)}$ be the obvious choice of norms for these spaces. As a consequence of the contractivity of $\iota: D_{\frac{p}{2}} \rightarrow H^p$, we can deploy the ingenious argument of Helson \cite{MR2263964} (also known as Bonami's lemma because of Lemma~1 in \cite{MR283496}) to show that the inclusion of $D_{\frac{p}{2}} (\D^d)$ in $H^p (\T^d)$ is also contractive for all $d \geq 2$. Since this technique is considered standard for the experts, we omit the details of the proof. 
        
        \begin{thm} \label{ExtFurtherDim}
            Let $d \geq 1$ and $p > 2$. Then we have that
                \[
                    \| f \|_{H^p (\T^d)} \leq \| f \|_{D_\frac{p}{2} (\D^d)}
                \]
                for all $f \in D_\frac{p}{2} (\D^d)$.
        \end{thm}

            The Hardy space of Dirichlet series $\h^p$ can be identified with $H^p (\T^\infty)$ by means of the Bohr lift \cite{MR1919645}, so from Theorem \ref{ExtFurtherDim} we deduce the inequality below.

        \begin{coro} \label{CoroDirichlet}
            If $p > 2$ and $n \geq 1$, we have that
                \begin{equation} \label{CoroDirichletDisp}
                    \left( \lim_{T \rightarrow \infty} \dfrac{1}{2 T} \int_{-T}^{T} \left| \sum_{k = 1}^n \dfrac{a_k}{k^{it}} \right|^p \: dt \right)^\frac{1}{p} \leq \left( \sum_{k = 1}^n |a_k|^2 d_\frac{p}{2}(k) \right)^\frac{1}{2},
                \end{equation}
                where $\left \{ d_\frac{p}{2} (n) \right \}_{n \geq 1}$ is the sequence of coefficients of $\zeta^\frac{p}{2}(s)$ as a Dirichlet series.
        \end{coro}

        It should be pointed out that it was already known that \eqref{CoroDirichletDisp} holds if $p = 2k$ \cite[Lemma~3]{MR3870953}, and interpolation methods showed that for all $p \in (2, \infty) \setminus \N$ there exists a $\beta > \frac{p}{2}$ such that the inclusion of $D_\beta(\D^\infty)$ in $H^p (\T^\infty)$ is contractive \cite{MR3078269}. Thus, Corollary \ref{CoroDirichlet} is a refinement of all these results.

    \section{Contractive inequalities for the Riesz projection}

        In this section we are going to expose the connection of Theorem \ref{MainTheo} with the classical Riesz's projection $P_+$. We recall that $P_+$ is defined as
            \[
                P_+ F (e^{it}) := \sum_{k = 0}^\infty \widehat{F}(k) e^{ikt}
            \]
            for all $F \in L^1 (\T)$.

        As a consequence of a celebrated result of Riesz \cite{MR1544909}, $P_+$ is bounded from $L^q (\T)$ to itself for every finite $q > 1$, a fact that was quantified by Hollenbeck and Verbitsky \cite{MR1780482} when they proved the following sharp inequality:
                \[
                    \| P_+ G \|_{H^q} \leq \csc \left ( \dfrac{\pi}{q} \right) \| G \|_{L^q}, \quad \forall G \in L^{q} (\T).
                \]
                Marzo and Seip \cite{MR2793046} showed that $P_+$ is a contractive operator from $L^q (\T)$ to $H^\frac{4q}{q + 2}$ if $2 \leq q \leq \infty$, and using duality arguments the same can be said from $L^q (\T)$ to $H^\frac{2q}{4-q}$ if $\frac{4}{3} \leq q \leq 2$ \cite{MR3858278}. However, the norm of $P_+ : L^q (\T) \rightarrow X$ is not known in general if $X$ is a space of analytic functions containing $H^q$. In the case that $X$ is a Hardy space, we can find in the literature the conjecture below.
        
        \begin{conjectureA}[Brevig, Ortega-Cerdà, Seip and Zhao \cite{MR3858278}] \label{StrongBOSZ}
            The Riesz projection $P_+$ is a contraction from $L^{p'} (\T)$ to $H^{\frac{4}{p}}$ for all $1 \leq p < \infty$, where $p' = \frac{p}{p-1}$. If $p = \infty$, we have that
                \[
                    \exp \left( \dfrac{1}{2\pi} \int_0^{2\pi} \log \left| P_+ F (e^{it}) \right| \: dt \right) \leq \| F \|_{L^1},
                \]
                for all $F \in L^1 (\T)$.
        \end{conjectureA}

        It was already known that if Conjecture \ref{StrongBOSZ} holds in the interval $(1, 2)$ then Conjecture \ref{BOSZ} is true \cite[Theorem 10]{MR3858278}. Thus, Theorem \ref{MainTheo} can be understood as new evidence in favour of Conjecture \ref{StrongBOSZ}. In fact, we finish this paper showing that Theorem \ref{MainTheo} has an actual application to the Riesz projection.

        \begin{coro}
            Let $p > 2$ and $\alpha > -1$. Then $P_+ : L^{p'} (\T) \rightarrow A^2_\alpha$ is  contractive if and only if $\alpha \geq \frac{p}{2} - 2$.
        \end{coro}

        \begin{proof}
            Consider the tests
                \begin{equation} \label{ExampleBOSZ}
                    F_\varepsilon (e^{it}) := \dfrac{1 - \varepsilon e^{it}}{(1 - \varepsilon e^{-it})^{1 - \frac{2}{p'}}}, \quad \varepsilon > 0.
                \end{equation}
                It is natural to consider this family of functions here because it shows that, if true, the exponent $\frac{4}{p}$ in Conjecture \ref{StrongBOSZ} cannot be improved \cite[Theorem~9]{MR3858278}. It is a straightforward computation to check that $\| P_+ F_\varepsilon \|_{A^2_\alpha} \leq \| F_\varepsilon \|_{L^{p'}}$ for all $\varepsilon > 0$ implies that $\alpha \geq \frac{p}{2} - 2$, and hence we have proved the necessity of this condition.
            
            Take $F \in L^{p'} (\T)$. Since the dual of $A^2_{\frac{p}{2} - 2}$ with respect to the $H^2$-pairing is $D_{\frac{p}{2}}$ and $P_+$ is self-adjoint, we have that
                \[
                    \left | \left \langle P_+ F, g\right \rangle_{H^2} \right | = \left | \left \langle F, g \right \rangle_{L^2 (\T)} \right | \leq \| F \|_{L^{p'}} \| g \|_{H^p} \leq \| F \|_{L^{p'}} \| g \|_{D_{\frac{p}{2}}},
                \]
                for all $g \in D_{\frac{p}{2}}$ and hence $P_+$ is contractive from $L^{p'}(\T)$ to $A^2_{\frac{p}{2}-2}$. Thus, Corollary \ref{CharacBergmanContractive} yields that
                \[
                    \| P_+ F \|_{A^2_\alpha} \leq \| P_+ F \|_{A^2_{\frac{p}{2} - 2}} \leq \| F \|_{L^{p'}}, \quad \forall F \in L^{p'} (\T),
                \]
                if $\alpha \geq \frac{p}{2} - 2$ and then we have completed the proof.
        \end{proof}

    \section*{Acknowledgements}

        The author wants to thank Professor K. Seip for the interesting discussions about the problem treated in this paper. He is also most grateful to A. Kouroupis, whose useful suggestions have made possible Section \ref{SectionDirichlet} of this paper. The author's work is funded by Grant 275113 of the Research Council of Norway through the Alain Bensoussan Fellowship Programme from ERCIM and is partially supported by grant PID2019-106870GB-I00 from Ministerio de Ciencia e Innovación (MICINN). 
    
    \bibliographystyle{amsplain} 
    \bibliography{Biblio}     

\end{document}